\documentclass[12pt, twoside, a4paper]{article}
\usepackage{etoolbox}
\usepackage[top=1.5in, bottom=1in, left=1in, right=1in]{geometry}
\usepackage{amsfonts, amsmath, amssymb, amsthm, amscd}
\usepackage{abstract}
\usepackage{lmodern}
\usepackage{fancyhdr}
\usepackage{mathtools}
\usepackage{mathrsfs} 
\usepackage{pst-node}
\usepackage{bigints}
\usepackage{mathrsfs}
\usepackage{color}
\usepackage{titlesec}
\usepackage[normalem]{ulem}

\usepackage{xcolor}
\usepackage{pagecolor}
\newcommand\redsout{\bgroup\markoverwith{\textcolor{red}{\rule[0.5ex]{2pt}{0.4pt}}}\ULon}

\usepackage{tikz-cd}

\usetikzlibrary{arrows}
\newtheorem{theorem}{Theorem}
\numberwithin{theorem}{subsection}
\numberwithin{theorem}{section}

\numberwithin{lemma}{section}
\newtheorem{proposition}{Proposition}
\numberwithin{proposition}{section}

\numberwithin{corollary}{section}
\numberwithin{equation}{section}
\numberwithin{equation}{subsection}

\newcommand{\RN}[1]{%
  \textup{\uppercase\expandafter{\romannumeral#1}}%
}

\allowdisplaybreaks

\newcommand{\R}{\mathbb{R}}

\newcommand{\Z}{\mathbb{Z}}

\newcommand{\bH}{\mathbf{H}}
\newcommand{\cH}{\mathcal{H}}

\newcommand{\tlambda}{\tilde{\lambda}}
\newcommand{\C}{\hat{C}}

\titleformat*{\section}{\large}
\titleformat{\subsection}[runin]
  {\normalfont\normalsize}{\thesubsection}{1em}{}

\titleformat{\subsubsection}[runin]
  {\normalfont\small}{\thesubsubsection}{1em}{}
\titleformat*{\paragraph}{\large\bfseries}
\titleformat*{\subparagraph}{\large\bfseries}
\newtheorem*{theorem-non}{Theorem}

\title{\normalsize \textbf{FLAT TRACE DISTRIBUTION OF THE GEODESIC FLOW ON COMPACT HYPERBOLIC PLANE}}
\author{Hy P.G. Lam \footnote{Research partially supported by NSF RTG grant DMS-1502632.}}
\date{}

\fancyheadoffset[R]{0cm}
\begin{document}
\maketitle

\begin{abstract}
    {\sc Abstract. } In this paper, we establish the spectral decomposition of the Koopman operator and determine the flat-trace distribution associated with the geodesic flow on the co-circle bundle over the compactification of Poincaré upper half-plane $\mathbf{H}^2 = \{z \in \mathbb{C} : \Im(z) > 0\}$, equipped with the hyperbolic metric $ds^2 = \frac{dz^2}{\Im(z)^2}$.
\end{abstract}
\section{\centering \sc Introduction}

Let $(M,g)$ be a compact Riemannian manifold with constant negative curvature equipped with a metric tensor $g$ and $S^*M$ be the cosphere bundle of $M$. The classical dynamics associated with the phase space $S^*M$ is the contact Hamiltonian flow 
\begin{equation*}
    G^t : (S^*M, \mu_L) \rightarrow (S^*M, \mu_L),~  G^t = \exp(tX_H)
\end{equation*} 
where $X_H \in C^\infty(S^*M, TS^*M)$ is the Hamiltonian vector field defined by $\alpha_{\text{can}}(X_H) = -dH$. 
Here, $\alpha_{\text{can}}$ denotes the canonical contact form on $S^*M$ induced from the projection $\pi: S^*M \rightarrow M$. Let $\hat{P} = i\mathcal{L}_{X_H}$ be the generator of the Koopman semigroup via 
\begin{equation*}
    \exp(-it\hat{P}) f = f  \circ G^t, ~~ t\geq 0
\end{equation*}
for $\mathbb{C}$-valued functions $f\in L^2(S^*M)$. The composition operator $V^t f  :=  f \circ G^t $ is called the \textit{Koopman operator}. We also write $\hat{L} = -i\hat{P}$ to denote the Liouville operator, namely the Lie-derivative with respect to $X_H$ 
\begin{equation*}
    \hat{L}f =  \{H, f\}
\end{equation*}
where $\{\cdot,\cdot\}$ denotes the Poisson bracket.  
The classical result by \cite{Ru} states that, for a surface of constant negative curvature, the geodesic flow on the unit (co)tangent bundle is a classic example of an Anosov flow, which exhibits uniform hyperbolic behavior - characterized by a continuous splitting of the tangent bundle into stable, unstable and flow directions
\begin{equation*}
    T_{\xi_x}S^*M = E^s(\xi_x) \oplus E^u(\xi_x) \oplus E^0(\xi_x),
\end{equation*}
where $ E^0(\xi_x) = \R X_H(\xi_x)$ and there exists constants $C, \theta$ such that for $v\in E^s(\xi_x)$ and $w \in E^u(\xi_x)$, 
\begin{align*}
    &||dG^t_{\xi_x}v ||_{g_S} \leq C e^{-\theta t}||v||_{g_S} \text{~ for all~} t \geq 0, \\
    &||dG^{-t}_{\xi_x}w ||_{g_S} \leq C e^{-\theta t}||w||_{g_S} \text{~ for all~} t \geq 0.
\end{align*}
In this setting, the geodesic flow satisfies the Lefschetz condition in which all of periodic orbits $\gamma$ associated with the length spectrum of $S^*M$ are non-degenerate. That is, the maps $I-\mathbf{P}_{\gamma}$ are invertible where $\mathbf{P}_{\gamma}$ denotes the linearized first-return map (see \cite{L} for more details). As a result of \cite{G}, the flat-trace distribution under this condition is the formula 
\begin{equation}
    \text{Tr}^\flat (V^t) = \sum_{\gamma \in \mathcal{P}} \frac{T^\#_\gamma}{|\det(I - \mathbf{P}_{\gamma})|}\delta(t- T_\gamma)
\end{equation}
where $\gamma$ runs over the set of primitive closed unit-speed geodesics. $T_\gamma$, $T^\#_\gamma$ are respectively the length and prime period of $\gamma$.

On the upper half-plane $\mathbf{H}^2 = \{z = x + iy \in \mathbb{C}: y > 0\}$ equipped with the hyperbolic metric 
$$
ds^2 = \frac{dx^2 + dy^2}{y^2},
$$
the unit tangent bundle $S\bH^2$ is the restriction of $T \bH^2$ to the submanifold of unit vectors under this metric. $S^*\bH^2$ is identified with $S\bH^2$ via the musical isomorphism. The group of orientation-preserving isometries on $\bH^2$ is identified with $G= PSL(2,\R)$ acting transitively by M\"{o}bius transformations. Given a point $(z,\xi)$ on $S^*\bH^2$, an element $g'$ from the isometry group acts by 
\begin{equation}\label{groupaction}
    g' \cdot (z, \xi) = (g' \cdot z , g'_{*z}(\xi)),
\end{equation}
where the differential $g'_* \in SO(2)$ acts on $\xi$ via rotation of the argument. Since the $0$-stabilizer acts faithfully on the unit fibre elements at each foot-point on the hyperbolic plane, \eqref{groupaction} defines a diffeomorphism between $S^*\bH^2$ and $G$. Namely, there exists a unique isometry $g$ such that $g \cdot (i, (0,1)) = (z, \xi)$. We identify $G/SO(2)$ with the hyperbolic plane whose metric induces a left-variant metric on $S^*\bH^2$ under the group multiplication.

Let $\Gamma$ be a full-rank lattice in $G$ and denote $\bH^2_\Gamma = \Gamma\backslash\bH^2$. The compactification of $\mathbf{H}^2$ is the hyperbolic quotient \begin{equation*}
    \Gamma\backslash G \equiv \Gamma\backslash S^*\bH^2 \equiv S^*\bH^2_\Gamma. 
\end{equation*} 
A hyperbolic element in $\gamma \in \Gamma$ is \textit{primitive} if $\gamma \neq \eta^n$ for any $\eta \in \Gamma$ and $n > 1$.  Each conjugacy class of a primitive hyperbolic element $[\gamma]$ uniquely determines a primitive closed geodesic in the $\bH^2_\Gamma$. 
\begin{theorem}\label{maintheoremmixing1}
Let $V_t$ be the Koopman operator associated with the geodesic flow acting on $L^2(\Gamma\backslash G)$
    \begin{equation}\label{flattrace1}
    \textup{Tr}^\flat V^t = \sum_{\gamma \in \mathcal{P}}\sum_{n =1}^\infty \frac{T^\#_\gamma}{2\sinh{(nT^\#/2)}}\delta(t- nT^\#_\gamma)
\end{equation}
where $\mathcal{P}$ is the set of primitive closed geodesics on $\Gamma\backslash G$, $nT^\#_\gamma$ is the length corresponding to the $n$-fold repetition of $\gamma$. The prime periods are 
$ T^\#_\gamma = \frac{4\pi}{r}$ where $r = \Im{l} , l = -\frac{1}{2} + ir$  associated with the eigenvalues of the Laplacian $\Delta_{\bH^2_\Gamma}$.
\end{theorem}



To prove Theorem \ref{maintheoremmixing1}, we invoke the representation theory of the projective special unitary group $PSU(1,1) = SU(1,1)/\mathbb{Z}_2$. Let $\Psi: G \rightarrow PSU(1,1)$ be an isomorphism. We realize the geodesic flow on the compact quotient space $\Psi(\Gamma)\backslash PSU(1,1)$. The decomposition of the Hilbert space $L^2(\Psi(\Gamma)\backslash PSU(1,1))$ into irreducible representations provides the expansion of $G^t$ as right-regular action on $PSU(1,1)$ in terms of unitary transformations.  In general, an eigenfunction $\varphi_\lambda \in L^2(S^*M)$ of the Koopman operator corresponding to an eigenvalue $\lambda$ satisfies 
$$
(V^t \varphi_\lambda) (\xi_x) = e^{\lambda t }\varphi_\lambda(\xi_x).
$$
For $f, \varrho \in C_c^\infty(S^*M)$ and any time $t \geq 0$, we define the \textit{correlation function} 
\begin{equation*}
    C(f, \varrho)(t) : = \int_{S^*M}V^tf(\xi)\overline{\varrho(\xi)} \mu_L(d\xi).
\end{equation*}
with the flow-invariant Leray form $\mu_L = \frac{(dx \wedge d\xi)^n}{dH}$. 
Hence, $V^t$ is unitary, that is, 
$$
C(\varrho, f)(t)= C(f,\varrho)(-t).
$$
We also define the resolvent of $\hat{L}$ to be the operator
\begin{equation*}
R(z)  := (\hat{L} -z )^{-1}: L^2(S^*M)\rightarrow L^2(S^*M), ~ \Re{z} > 0,
\end{equation*}
which has meromorphic continuation to $\mathbb{C}$ and poles at the spectral values $\lambda$. The associated correlation function of $R(z)$ is 
\begin{equation*}
    \C(f,\varrho)(z) := \int_{S^*M} R(z)f(\xi) \overline{\varrho (\xi)} \mu_L(d\xi),
\end{equation*}
 which equates to 
\begin{equation}\label{integralformhatC}
   \int_0^\infty \int_{S^*M} e^{(\hat{L}-z)s}f(\xi)\overline{\varrho(\xi)} \mu_L(d\xi) ds = \int_0^\infty e^{-zs}C(f, \varrho)(s) ds. 
\end{equation}
In the case $M = \bH^2_\Gamma$, the expansion of the regular transformation allows $\hat{C}(\cdot, \cdot)(z)$ to be written as the sum 
\begin{equation*}
\hat{C}(\cdot, \cdot)(z) = \sum_{\lambda}\frac{\langle \cdot , \varphi_\lambda \rangle_{L^2} \langle \cdot, \varphi_{-\bar{\lambda}} \rangle_{L^2} }{\lambda -z} = 2\pi i \sum_{\lambda} \frac{\textup{Res}(\lambda, \hat{C})}{\lambda -z}
\end{equation*}
on a dense subset of the rapidly decaying test functions (see \eqref{theset} and Proposition \ref{decompofresolvent}). This further leads to a complete spectral resolution of $V^t$ as a distributional kernel. 
 \begin{theorem}\label{maintheoremmixing2}
 Consider the geodesic flow $G^t$ realized on the compact quotient space $\Psi(\Gamma)\backslash PSU(1,1)$. Let $V(t)$ be the Koopman operator associated with this flow acting on $C^\infty (\Psi(\Gamma)\backslash PSU(1,1) )$. As an operator in $\mathcal{D}'(\Psi(\Gamma)\backslash PSU(1,1))\otimes \mathcal{D}'(\Psi(\Gamma)\backslash PSU(1,1))$, 
\begin{equation*}
    V(t) = \sum_{l \in \mathcal{S}(\Gamma)} \sum_{\lambda \in \sigma(l)} \sum_{j = 1}^{\mu(\Psi(\Gamma))(l)} e^{\lambda t}  \varphi^*_{\lambda, j} \otimes \varphi^\dagger_{-\bar{\lambda}, j} 
\end{equation*}
where \begin{align*}
&\bullet \mathcal{S}(\Gamma) =\left\{ l  \in \mathbb{C} :  l= -\frac{1}{2} + ir_j, \Im{l} = r_j  \geq 0; ~ 
l \in (-1,0)\backslash\left\{-\frac{1}{2}\right\};~ l = n  - \frac{1}{2}, n \in \mathbb{Z}_{> 0}\right \}, \\
&\bullet \sigma(l ) =\{l - k: k \in \mathbb{Z}_{\geq 0}\}. 
\end{align*}
\end{theorem}
The prime period values in \eqref{flattrace1} are contained in the discrete set $S(\Gamma)$ (see \eqref{valuesofl}) corresponding to the representation of the first principal series.

\subsection*{\textbf{Acknowledgment. }} This article forms part of the author’s PhD thesis at Northwestern University. The author wishes to express his gratitude to the late Steve Zelditch for many insightful discussions and guidance on the topic of Koopman dynamics on general manifolds.

\section{\centering \sc Preliminaries}

\subsection{\textbf{Flat-trace distribution}.}
The Schwartz kernel of the Koopman operator is a $\delta$-distribution with compact support in $\textup{Graph}(G^\bullet) \subset S^*M \times S^*M \times \mathbb{R}^+$ given by
\begin{equation*}
    V^\bullet f(\xi) = \int_{S^*M} f(\eta) \, \delta(\eta - G^\bullet \xi) \, d\mu_L(\eta),
\end{equation*}
(see \cite{G}, \cite{La} for background).
Consider the diagonal 
\begin{equation*}
    \Delta : S^*M \times \mathbb{R}^+ \to S^*M \times S^*M \times \mathbb{R}^+, \quad (\zeta, t) \mapsto (\zeta, \zeta, t),
\end{equation*}
and the projection map
\begin{equation*}
    \Pi : S^*M \times \mathbb{R}^+ \to S^*M, \quad (\zeta, t) \mapsto \zeta.
\end{equation*}
Assuming $\Delta$ intersects the graph of $G^\bullet$ transversally, the flat-trace distribution is a well-defined $\delta$-distribution on $\mathbb{R}^+$ given by the pushforward-pullback 
\begin{equation} \label{second}
    \textup{Tr}^\flat V^\bullet = \Pi_* \Delta^* k^\bullet,
\end{equation}
where $k^t(\xi, \eta) = \delta(\xi - G^t \eta)$. For fixed $t\geq 0$, the formal integral formula of $\textup{Tr}^\flat V^\bullet $ is 
\begin{equation}\label{flatdef}
    \textup{Tr}^\flat V^t = \int_{S^*M}\delta(\xi - G^t\xi) d\mu_L(\xi).
\end{equation}

\subsection{\textbf{The hyperbolic Laplacian}.}
In hyperbolic polar coordinates centered at the origin, the metric on $\bH^2$ becomes
$$
ds'^2  = dr'^2 +\sinh{r'}^2 d\theta^2.
$$
The positive Laplacian with the respect to this metric is 
$$
\Delta_{\bH^2}= \partial^2_{r'r'} + \coth{r}\partial_{r'} + \frac{1}{\sinh{r'}^2}\partial^2_{\theta\theta}
$$
with $\Delta_{\bH^2}$-eigenvalues notated by 
\begin{equation}\label{hyperboliceigenval}
   \lambda'^2 = \begin{cases}   \lambda'^2_r = \frac{1}{4} + r^2\\\lambda'^2_{s'} = s'(1 - s'), ~ s'= \frac{1}{2} + ir\end{cases}.
\end{equation}

The joint-eigenfunctions of $\Delta_{\bH^2}$ and of the rotational group $K$ are the spherical functions
\begin{equation*}
    \begin{cases}
        \Delta_{\bH^2}\Phi^m_{s'} &= \lambda'^2_{s'} \Phi^m_{s'}\\
        -i \partial_\theta \Phi^m_{s'} &= m \Phi^m_{s'}
    \end{cases}.
\end{equation*}


In the representation of $L^2(\Gamma \backslash G)$, the Casimir operator (see \S \ref{casimir}) is closely related to the hyperbolic Laplacian $\Delta_{\bH^2}$.


\subsection{\textbf{Dynamics and representation theory on the hyperbolic quotient}.}\label{reptheoG}
The  tangent space $TS^*M$ splits into orthogonal subspaces
$$
T_{G^t}S^*M = N^h_{G^t}\oplus N^v_{G^t} \oplus \text{span}_{\R}\{\dot{G^t}\}.
$$
The horizontal and vertical (fibre) subspaces are respectively spanned by the generators $P, Q$  and $W$ of the Lie algebra $\mathfrak{sl}(2, \R)$ which are
$$ 
P = \begin{pmatrix}
    1 & 0 \\
    0 & -1
\end{pmatrix}, ~ Q = \begin{pmatrix}
    0 & 1 \\
    1 & 0
\end{pmatrix}, ~ W = \begin{pmatrix}
    0 & -1\\
    1 & 0 
\end{pmatrix}.
$$
Additionally,
$$ 
\exp(tP) = \begin{pmatrix}
    e^t & 0 \\
    0 & e^{-t}
\end{pmatrix}, ~ \exp(tQ) = \begin{pmatrix}
    \cosh{t} & \sinh{t} \\
    \sinh{t} & \cosh{t}
\end{pmatrix}, ~ \exp(tW) = \begin{pmatrix}
    \cos{t} & -\sin{t}\\
    \sin{t} & \cos{t} 
\end{pmatrix}.
$$
The matrices satisfy the commutator relations 
$$
[P, Q] = 2W, ~ [P,W] = 2Q, ~ [Q,W] = 2P, 
$$
which forms the structure of $\mathfrak{sl}(2,\R)$. 

The Iwasawa-decomposition of $G = PSL(2,\R)$ is $G = KAN $ where 
\begin{align*}
    K &= \left\{ k_\theta =  \begin{pmatrix}
    \cos{\theta} & -\sin{\theta}\\
    \sin{\theta} & \cos{\theta} 
\end{pmatrix} \in SL(2,R) : \theta \in [0, 2\pi) \right \} \simeq SO(2),\\
    ~A &= \left\{ a_t = \begin{pmatrix}
    e^{t/2} & 0 \\
    0 & e^{-t/2} 
\end{pmatrix} \in SL(2,\R): y > 0   \right \},\\
 N &= \left\{ n_u =  \begin{pmatrix}
    1 & u \\
    0 & 1
\end{pmatrix} \in SL(2,R):  u\in \R \right\}.
\end{align*}

In the compactification of $\bH^2$, geodesic flow on the quotient phase space $\Gamma\backslash G$ is the right-action 
$$\begin{cases}
    G^t \cdot \Gamma g = \Gamma g a_t, ~~ a_t 
\in A; \\
ga_s( \theta, t,u ) = g(\theta , t+s, ue^{-s}),~ s \in \R_{\geq  0} 
\end{cases}, $$
and the horocycle flow is defined by 
$$\begin{cases}
    h^u \cdot \Gamma g = \Gamma g n_u, ~~ n_u \in N \\
    gn_{u'}(\theta, y,u) = g(\theta, t,  u+u') ,~ u , u' \in \R
\end{cases}.$$
The Koopman operator acts on $L^2(\Gamma\backslash G)$ via the composition $V^t f (g \Gamma ) = f ( ga_t\Gamma)$. 


\subsection{\textbf{Decomposition of $L^2(\Gamma \backslash G)$}.}\label{casimir}
Denote $\cH : =L^2(\Gamma \backslash G)$ and $ U(\cH)$ to be the group of unitary transformations of $G$ on $\cH$. For any unitary representation $T : G \rightarrow U(\cH) $ on $G$ where  $T_g(f) (\Gamma h) = f (\Gamma hg)$, the infinitestimal Lie-algebra action induced by $T$ is the Lie derivative 
$$
\mathcal{L}_X  = \lim_{t\rightarrow 0}t^{-1}(T_{\exp(tX)} - I),~ X \in \mathfrak{sl}(2,\R).
$$
For $\Gamma$-periodic functions $f, \varrho$ in $C^\infty_c(G)$, 
$$
\langle \mathcal{L}_X f, \varrho \rangle = - \langle f, \mathcal{L}_X \varrho \rangle. 
$$
The densely-defined Casimir operator on $\cH(T)$ is
$4 \Omega_T =  \mathcal{L}_P^2 + \mathcal{L}_Q^2  - \mathcal{L}_W^2$ whose closure $\overline{\Omega}_T$ is self-adjoint.  We also mention the differential operators (see \eqref{lowraise})
\begin{equation}\label{Liebracket}
N_-:= \mathcal{L}_W + \mathcal{L}_Q, ~ N_+ := \mathcal{L}_W - \mathcal{L}_Q,
\end{equation}which satisfy
\begin{equation*}
    [L, N_\pm] = \pm N_\pm, , ~ [N_+, N_-]= -2\mathcal{L}_P. 
\end{equation*} 

For any $X$, 
$$[\Omega_T, \mathcal{L}_X]\big|_{\mathcal{D}(\Gamma\backslash G)} = 0$$
on the space of test functions. By the Baker-Campell-Hausdorff formula, the unitary transformation $T_g$ can be written as compositions of the exponentials of the Lie algebra elements. For small $t$, the group action of the Lie algebra can be expressed via the exponential map $T_{\exp(tX)} = e^{t\mathcal{L}_X}$. Thus, $[\overline{\Omega}_T, T_g] = 0$ for any $g\in G$. If $T$ is irreducible, the commutativity implies that 
$\Omega_T$ acts on $\cH(T)$ via scalar multiplication
\begin{equation*}
    \left(\Omega_T   + \lambda'^2 I \right)\big|_{\cH(T)} = 0
\end{equation*}
 by Schuys lemma. Precisely, $ \lambda'^2 (r) = \frac{1}{4}  + r^2$ where $r$ is the spectral parameter that corresponds with the eigenvalues of  $\Delta_{\bH^2}$ in \eqref{hyperboliceigenval} (cf. 
 \cite{Kn}).
For each $r$, $\cH(T^r)$ is a direct sum of at-most-1-dimensional representations $\cH(T^{r,m})$ associating with weights $m \in \mathbb{Z}$. Namely, the $\Gamma$-invariant joint $(r,m)$-eigenfunctions 
\begin{equation}\label{rmeigenmodes}
    \begin{cases}
    \Omega_{T^r} v_{r,m} &=  -\left(\frac{1}{4} + r^2\right) v_{r,m}\\
    e^{\mathcal{L}_W} v_{r,m} &= e^{im} v_{r,m} 
\end{cases}
\end{equation}
of $\Omega_{T^r}$ and the generator $W$ are the automorphic functions that span the invariant subspaces $\cH(T^{r,m})$ hence form a basis for $\cH(T^r)$ (cf. \cite{Bo}). Let $\Delta_{\bH^2_\Gamma}$ be the induced Laplacian on the quotient space $\bH^2_\Gamma$ with $\sigma(\Delta_{\bH^2_\Gamma})$ denoting its discrete spectrum (see \cite{L}, \cite{Kn} for treatise on the Maass cusp forms in the case $\Gamma = SL(2,\Z)$). Since the vectors $v_{r,0}$ are $K$-invariant, 
$$
\Omega_{T^{r,0}}  = \Delta_{\bH^2_\Gamma}.
$$
We have the decomposition into the irreducible representations 
\begin{align}
    \cH 
    &= \bigoplus_{-\frac{1}{4} - r^2 \in \sigma(\Delta_{\bH^2_\Gamma})} \bigoplus_{m \in \mathbb{Z} }  \mu_\Gamma(r)\cH(T^{r,m}). \label{decomp1}
\end{align}
We write $\mu_\Gamma(r)\cH(T^{r,m})$ $=$ $\bigoplus_{j=1}^{\mu_\Gamma(r)}\cH(T^{r,m,j})$ where $\mu_{\Gamma}(\cdot)$ denotes the multiplicity of the corresponding eigenvalue of $\Delta_{\bH^2_\Gamma}$. We later replace $r$ with parameter $l = -\frac{1}{2} + ir$ for the analysis of the irreducible representations on $SU(1,1)$ (see \S \ref{SU(1,1)}).


\subsection{\textbf{Unitary representations of the principal and discrete series}.}\label{SU(1,1)}
We consider $SU(1,1)$ - the group of quasi-unitary unimodular matrices 
$$
\bigg \{g = \begin{pmatrix}\alpha & \beta \\ \bar{\beta} & \bar{\alpha} \end{pmatrix} \in SL(2,\mathbb{C}) : |\alpha|^2 - |\beta|^2 = 1, g\eta g^{\dagger} = \eta\bigg\},
$$
where $\eta = \begin{pmatrix}
    1 & 0 \\ 0 & -1
\end{pmatrix}$ is the normal form of index-2 nonsingular Hermittian matrices. Furthermore, $g$ can be parametrized by the Euler angles via the product representation
\begin{align}
    g (\phi, \tau, \psi) &= \begin{pmatrix}
        e^{\frac{i\phi}{2}}& 0 \\
        0 & e^{-\frac{i\phi}{2}}
    \end{pmatrix}\begin{pmatrix}
        \cosh{\frac{\tau}{2}} & \sinh{\frac{\tau}{2}} \\
        \sinh{\frac{\tau}{2}} & \cosh{\frac{\tau}{2}}
    \end{pmatrix} \begin{pmatrix}
        e^{\frac{i\psi}{2}} & 0 \\
        0& e^{-\frac{i\psi}{2}}
    \end{pmatrix} \nonumber \\
    &= \begin{pmatrix}
        \cosh{\frac{\tau}{2}}e^{\frac{i(\phi+\psi)}{2}} & \sinh{\frac{\tau}{2}}e^{\frac{i(\phi -\psi)}{2}}\\
        \sinh{\frac{\tau}{2}}e^{\frac{i(\psi -\phi)}{2}}& \cosh{\frac{\tau}{2}}e^{\frac{-i(\phi+\psi)}{2}}
    \end{pmatrix}
\end{align}
with $0 \leq  \phi \leq 2\pi,  -2\pi \leq \psi \leq 2\pi$, $ 0 \leq \tau < \infty$. The space $L^2( SU(1,1))$ is defined by the inner product  
$$
\langle f, \varrho  \rangle  = \frac{1}{8\pi^2} \int_{\R_{\geq 0}}\int_0^{2\pi}\int_{-2\pi}^{2\pi}f(g(\phi, \tau, \psi) \overline{\varrho(g(\phi, \tau, \psi))}\sinh{\tau}d\psi d\phi d\tau.
$$
with respect to the Haar measure $dg = \sinh{\tau}d\tau d\psi d\phi$. By periodicity, an $L^2$-function $f$ has the Fourier expansion
\begin{equation}\label{Fourierseries}
    f(\phi, \tau, \psi) = \sum_{(m,n) \in \mathbb{Z}\times\mathbb{Z} \cup \left(\mathbb{Z} + \frac{1}{2} \right) \times  \left(\mathbb{Z} + \frac{1}{2}\right) } \Pi_{mn}(\tau) e^{-im\phi}e^{-in\psi}
\end{equation}
where $\Pi_{mn}(\tau) = \frac{1}{8\pi^2}\int_{-2\pi}^{2\pi}\int_0^{2\pi}f(\phi, \tau, \psi)e^{im\phi}e^{in\psi} d\phi d\psi$. 

\bigskip
%


Given $\chi = (l, \epsilon )$ with $l\in \mathbb{C}$, $\epsilon \in \{0, \frac{1}{2}\}$, the representations $T^\chi$ indexed by $\chi$ are realized on the Hilbert space (see \cite{Kn}, \cite{L}, \cite{V} for background)
$$\mathcal{D}_\chi : = \{\Phi \in C^\infty( \mathbb{C}; \mathbb{C}) : \Phi \text{~ is homogeneous of degree $2l$ and of parity $2\epsilon$} \}.$$
For $g' = \begin{pmatrix}
    a & b \\ c & d
\end{pmatrix} \in SL(2,\mathbb{R})$, the regular group-transformation on $\mathcal{D}_\chi$ is 
\begin{equation*}
    T_{g'}\Phi(z) = \Phi\bigg( \left(\frac{a + d}{2} + i \frac{b - c}{2}\right)z  + \left( \frac{b + c}{2} + i\frac{d - a}{2}\right)\bar{z}\bigg).
\end{equation*}
A unique quasi-unitary element $g$ in $SU(1,1)$ is obtained from $g'$ by the conjugation $ g= h'g'h'^{-1}$ where $h' = \frac{1}{\sqrt{2}}\begin{pmatrix}
    1 & i \\ i & 1 
\end{pmatrix}$, whose regular transformation is 
\begin{equation*}
    T_g\Phi(z) = \Phi(az - \overline{b}\overline{z}).
\end{equation*}
Consider a polar curve $\gamma$ in the complex plane with parametrization $\gamma(\theta)=  ae^{(i+b)\theta}$ with positive constants $a,b$ and $\theta \in (0, 2\pi)$, we fix any $z\in \mathbb{C}$ and let $z_0$ be the point of intersection between $\gamma$ and a straight line passing through the origin and $z$. For $\Phi \in \mathcal{D}_\chi$, 
$$
\Phi(z) = \Phi(z_0) \big|\frac{z}{z_0}\big|^{2l} \left(\frac{\hat{z}}{\hat{z_0}}\right)^{2\epsilon}
$$
where $\hat{z}$, $\hat{z_0}$ are the normalizations of $z,z_0$. Thus, $\mathcal{D}_\chi$ can be realized as the space of holomorphic functions on $\gamma$. If we set $a = 1, b = 0$, then $\gamma$ is the unit circle whereby 
we can associate with $\Phi$ a holomorphic function $f$ defined by 
$f(e^{i\theta}) = \Phi(e^{i\frac{\theta}{2}})$
if $\epsilon = 0$, and  
$f(e^{i\theta}) = e^{i\frac{\theta}{2}}\Phi(e^{i\frac{\theta}{2}})$ if $\epsilon = 1/2$. 
The representations of $SU(1,1)$ is then realized on the space $\mathcal{D} : = C^\infty(\mathbb{S}^1; \mathbb{C})$ by the action 
\begin{equation}
    T^\chi \begin{pmatrix}
        \alpha & \beta \\
     \bar{\beta}&   \bar{\alpha} 
    \end{pmatrix} f(e^{i\theta})= (\beta e^{i\theta} + \bar{\alpha})^{l +\epsilon} (\bar{\beta}e^{-i\theta} + \alpha)^{l- \epsilon}f\left(\frac{\alpha e^{i\theta} + \bar{\beta}}{\beta e^{i\theta} + \bar{\alpha}}\right). \label{T^chiaction}
\end{equation}
For values of $l$ and $\epsilon$ such that $l\pm \epsilon \notin \mathbb{Z}$, $T^\chi$ is irreducible.  

Under the transformation of $\bH^2$ into $\mathbb{S}^1$ via the map 
$$
z \mapsto \frac{z-i}{z+i},
$$
one can construct an isomorphism $\Psi: SL(2,\R)  \rightarrow SU(1,1)$. $SO(2)$ corresponds to the subgroup of diagonal matrices 
$$
\Omega = \left\{ \begin{pmatrix}
    e^{it/2} & 0 \\ 0 & e^{-it/2}
\end{pmatrix} \in SL(2,\mathbb{C}): t \geq 0 \right\} \lhd SU(1,1)
$$
under $\Psi$. Hence, similar to the decomposition of $\mathcal{H}(T^r)$ in \eqref{decomp1} is the orthogonal decomposition of invariant subspaces of $\mathcal{D}$ into the at-most-1-dimensional spaces $\mathcal{H}_k$ spanned by $\{e^{-ik\theta}\}$ where $k \in  \mathbb{Z} \cup \mathbb{Z} + \frac{1}{2} $.  Consequently, the irreducible representations of $SL(2,\R)$ are obtained from that of $SU(1,1)$ and vice versa.


The matrix elements for the representation $T^\chi_g$ are the Fourier coefficients in 
\begin{equation}\label{Fourier}
    T^\chi_g e^{-in\theta} = \sum_{m \in \mathbb{Z} \cup \mathbb{Z} + \frac{1}{2}} (T^\chi_g)_{mn} e^{-im\theta}
\end{equation}
where 
\begin{equation}
   (T^\chi_g)_{mn} = \frac{1}{2\pi}\int_0^{2\pi}(\beta e^{i\theta} +\bar{\alpha})^{l+n + \epsilon}(\bar{\beta}e^{-i\theta} +\alpha)^{l-n -\epsilon}e^{i(m-n)\theta}d\theta.
\end{equation}
Consider the element 
$$
g_\tau = \begin{pmatrix}
    \cosh{\frac{\tau}{2}} & \sinh{\frac{\tau}{2}} \\
    \sinh{\frac{\tau}{2}} & \cosh{\frac{\tau}{2}}
\end{pmatrix},
$$
the matrix elements for the unitary transformation $T^\chi_{g_\tau}$ are 
\begin{equation}
  (T^\chi_{g_\tau})_{mn} = \frac{1}{2\pi}\int_0^{2\pi} \left(\cosh{\frac{\tau}{2}} + \sinh{\frac{\tau}{2}}e^{i\theta} \right)^{l+n + \epsilon}  \left(\cosh{\frac{\tau}{2}} + \sinh{\frac{\tau}{2}}e^{-i\theta} \right)^{l - n - \epsilon}e^{i(m-n)\theta}d\theta
\end{equation}
We can write 
\begin{equation}\label{g_taumatrixele}
    (T^\chi_{g_\tau})_{mn}  = B^l_{m'n'}(\cosh{\tau}) 
\end{equation}
where $B^l_{m'n'}$ is the Jacobi function 
\begin{align*}
B^l_{m'n'}(\cosh(\tau)) &= \frac{1}{2\pi}\int_0^{2\pi} \left(\cosh{\frac{\tau}{2}} + \sinh{\frac{\tau}{2}}e^{i\theta} \right)^{l +n'}  \left(\cosh{\frac{\tau}{2}} + \sinh{\frac{\tau}{2}}e^{-i\theta} \right)^{l - n' }e^{i(m'-n')\theta}d\theta,
\end{align*}
$\tau \in \R_{\geq 0}, l \in \mathbb{C}, m,n \in \mathbb{Z},$  $m' = m + \epsilon$, $n' = n +\epsilon$. Denote by $h_{\phi}$, $h_\psi$ the diagonal matrices that correspond to the first (precession) and third (intrinsic) rotations. Applying \eqref{T^chiaction} yields 
\begin{equation}\label{flowonS^1}
T^\chi_{h_t} f(e^{i\theta}) = e^{-i\epsilon t}f(e^{i(\theta + t)})
\end{equation}
where $h_t \in \Omega$.
The matrix representations of the transformations $T^\chi_{h_\phi}$, $T^\chi_{h_\psi}$ are the diagonal matrices whose elements are 
$(T^\chi_{h_t})_m  =e^{-im't}$, $t = \phi, \psi$.
Therefore, 
\begin{align}
      (T^\chi_{g})_{mn} &= (T^\chi_{h_{\phi}}T^\chi_{g_\tau}T^\chi_{h_\psi})_{mn} = e^{-im'\phi}e^{-in'\psi}B^l_{m'n'}(\cosh{\tau}) \label{matrixele}
\end{align}
for a general element $g \in SU(1,1)$. Given any $g_1, g_2$, the matrix elements are
\begin{align*}
   ( T^\chi_{g_1g_2})_{mn} = \sum_{k =-\infty}^\infty (T^\chi_{g_1})_{mk} (T^\chi_{g_2})_{kn}.
\end{align*} 
In terms of the Euler angles, for $g_1 = g(0,\tau_1 , 0) , g_2= g(\phi_2, \tau_2, 0) $, 
\begin{align}\label{matrixele2}
    (T^\chi_{g_1})_{mk} = B^l_{m'k'}(\cosh{\tau_1}),~ (T^\chi_{g_2})_{kn} = e^{-ik'\phi_2}B^l_{k'n'}(\cosh{\tau_2}).
\end{align}

Infinitesimal actions from the Lie algebra $\mathfrak{qu}(2)$ are captured by the 1-parameter subgroups generated by matrices of forms 
\begin{equation*}
\omega_1(t) = \begin{pmatrix}
    \cosh{\frac{t}{2}} & \sinh{\frac{t}{2}} \\
     \sinh{\frac{t}{2}} & \cosh{\frac{t}{2}}
\end{pmatrix},~ \omega_2(t) = \begin{pmatrix}
    \cosh{\frac{t}{2}} & i\sinh{\frac{t}{2}} \\
    -i\sinh{\frac{t}{2}} & \cosh{\frac{t}{2}}\end{pmatrix},~ \omega_3(t) = \begin{pmatrix}
        e^{\frac{it}{2}} & 0 \\
        0 & e^{\frac{-it}{2}}
    \end{pmatrix}. 
\end{equation*}
The associated infinitesimal operators are the derivatives 
$$\hat{A}_{\omega_i} : = \frac{d}{dt}\big|_{t = 0}R_{\omega_i(t)} = \phi'(0)\partial_{\phi} + \tau'(0)\partial_\tau + \psi'(0)\partial_\psi $$ of the regular representation on $L^2(SU(1,1))$ via the right-action $R_{\omega_i(t)} F(g\Psi(\Gamma))$ $=$  \\
$F(g\omega_i(t)\Psi(\Gamma))$. Explicitly,
\begin{align*}
    \hat{A}_{\omega_1} &= \frac{\sin{\psi}}{\sinh{\tau}}\partial_{\phi} - \cos{\psi}\partial_\tau - \coth{\tau}\sin{\psi}\partial_\psi, \\
    \hat{A}_{\omega_2} &= \frac{\cos{\psi}}{\sinh{\tau}}\partial_\phi -\sin{\psi}\partial_\tau - \coth{\tau}\cos{\psi} \partial_\psi, \\
    \hat{A}_{\omega_3} &= \partial_\psi. 
\end{align*}
The analogous Casimir operator on $SU(1,1)$ is
\begin{equation}
    \Omega_{SU(1,1)} = - \hat{A}_{\omega_1}^2 - \hat{A}_{\omega_2}^2 + \hat{A}_{\omega_3}^2. 
\end{equation}
For each $f \in \mathcal{D} \subset L^2(\mathbb{S}^1;\mathbb{C})$ and index $m$, the $2$-to-$1$ mapping
\begin{equation}\label{corresDtoL^2}
    (T^\chi_\bullet, f ; m) \mapsto  \langle T^\chi_\bullet f, e^{-im\theta}\rangle = \frac{1}{2\pi}\int_0^{2\pi}T^\chi_\bullet f(e^{i\theta})e^{im\theta}d\theta
\end{equation}
sends $f$ to a square-integrable function on $SU(1,1)$. Denote by $\mathscr{H}^\chi_m \subset L^2(SU(1,1))$ to be the space spanned by functions of this form whose basis are the matrix elements $\{(T^\chi_\bullet)_{mn}\}_{n\in \mathbb{Z} \cup \mathbb{Z} +\frac{1}{2}}$ in \eqref{matrixele}. 
These elements are  the $\Psi(\Gamma)$-invariant $(l,m)$-eigenfunctions 
\begin{equation}\label{eigenfunctionSU(1,1)}
    \begin{cases}
    \Omega_{SU(1,1)} (T^\chi_\bullet)_{mn} &= l(l+1)(T^\chi_\bullet)_{mn} \\
    \hat{H}_+^k(T^\chi_\bullet)_{mn} &= \prod_{j =0 }^{k-1}(n'- l +j)(T^\chi_\bullet)_{m,n+j}\\
    \hat{H}_-^k(T^\chi_\bullet)_{mn} &= \prod_{j =0 }^{k-1} (-1)^k (n' + l -j)(T^\chi_\bullet)_{m,n-j}
\end{cases}
\end{equation}
where $\hat{H}_\pm$ are the raising/lowering operators  $\hat{H}_\pm = -\hat{A}_{\omega_1} \mp i\hat{A}_{\omega_2}$. 
For any values $l$ of the form $l = -\frac{1}{2} + ir$ with $r \in \R$ or $l \in \R$, the associated $T^\chi_g$-invariant Hermitian forms on $\mathcal{D}$ are positive-definite, hence $T^\chi_g$ are unitary representations. With $\chi =\left( -\frac{1}{2} + ir, 0 \right), \left( -\frac{1}{2} + ir, \frac{1}{2} \right)$ and  $\chi = (l, \epsilon)$ for $l \in \mathbb{Z}\cup \mathbb{Z}+\frac{1}{2}$,  $T^\chi_g$ are respectively the unitary representations of the \textit{first}, \textit{second} principal and the discrete series. 
On the quotient group $G \cong PSU(1,1) = SU(1,1)/\{\pm I\}$, regular right-transformation satisfies
$$
R_{-g_0}F(g) = R_{g_0}F(g).
$$
for any $g_0 \in SU(1,1)$, $F \in L^2(PSU(1,1))$ hence parity of the target functions in \eqref{corresDtoL^2} is $0$. Suppose $F \in \mathscr{H}^\chi_m$, since 
$$R_g\big|_{\mathscr{H}^\chi_m} F =  \langle T^\chi_{\cdot g} f, e^{-im\theta} \rangle  \equiv T^\chi_g  F$$ 
for $\epsilon = 0 , \frac{1}{2}$, the irreducible representations of the Hilbert space $\cH$ are then indexed by $\chi = (l, 0)$ for non-integral values $l$.
Denote by  
\begin{equation*}
  (T^\chi_g)_{mn} =  \begin{cases}
         (T^\chi_g)^+_{mn} \text{~for~} m > 0, n >0\\
          (T^\chi_g)^-_{mn} \text{~for~} m<0, n < 0
    \end{cases}
\end{equation*}
the matrix elements of unitary representations of the discrete series. 
For $F$ a square-integrable function on $PSU(1,1)$, we have the decomposition into the first principal and the discrete series 
\begin{align}
    F(g) &= \frac{1}{4\pi^2}\sum_{m,n}\left[\int_{0}^\infty a_{mn}(r) \left(T^{(-\frac{1}{2} + ir, 0)}_g\right)_{mn} r \tanh (\pi r)dr + \sum_{l = 1}^M \left(l - \frac{1}{2}\right) b^\pm_{mn} (l)(T^{(l,0)}_g)^\pm_{mn} \right] \label{decompl2qu2}
\end{align}
where the summation runs over integral values of $m,n,l$, $M : = |m|\wedge |n|$ and 
\begin{align*}
    a_{mn}(r) &= \int_{SU(1,1)} F(g) \overline{\left(T_g^{\left(-\frac{1}{2} + ir, 0\right)}\right)_{mn}} dg, \\
    b^\pm_{mn}(l) 
    &= \frac{(-1)^{m-n}\Gamma(l+m+1)\Gamma(l - m +1)}{\Gamma(l+n+1)\Gamma(l -n +1)}\int_{SU(1,1)} F(g)\overline{(T^{(l,0)}_g)^\pm_{mn}}  dg.
\end{align*}
Since each element in $\mathscr{H}^{(l,0)}_m$ can be written as the form
$$
F(g)  = \sum_{n \in \mathbb{Z}}C_{mn}(l)(T^{(l,0)}_g)_{mn}
$$
where $C_{mn}$ are the coefficients according to \eqref{decompl2qu2}. 
By unitarity, the $L^2$-norm of $F$ is given by the Plancherel formula, \begin{align*}
\int_{SU(1,1)} |F(g)|^2dg &= \frac{1}{4\pi^2}\sum_{m,n}\bigg[\int_0^\infty |a_{mn}(r)|^2r \tanh{(\pi r)}dr + \sum_{l = 1}^M\left(l - \frac{1}{2}\right) |b^\pm_{mn} (l)|^2 \bigg]
\end{align*}
The expansion of the regular transformation on $PSU(1,1)$ is 
\begin{align}
     R_g &= \frac{1}{4\pi^2}\sum_{m\in \mathbb{Z}}\left[\int_{0}^\infty a_{mn}(r) \left(T^{(-\frac{1}{2} + ir, 0)}_g\right)_{mn} r \tanh (\pi r)dr + \sum_{l = 1}^M \left(l - \frac{1}{2}\right) b^\pm_{mn} (l)(T^{(l,0)}_g)^\pm_{mn} \right] \label{decompl2qu2(2)}
\end{align}
As $\mathscr{H}_m^{(l,0)} \subset L^2(PSU(1,1))$ is right-invariant under the regular transformation via the 1-parameter subgroup $\Omega$, without loss of generality, the functions $F_{m,l}:= \langle T^{l,0}_\bullet f, e^{-im\theta} \rangle$ has the $\Psi(\Gamma)$-periodic form 
\begin{equation}\label{unwrapping}
F_{m,l} (\phi, \tau, \psi) = \bar{F}_{m,l}(\phi, \tau)e^{-im\psi}, 
\end{equation}
$$
\int_{D \subset [0,2\pi]\times \R_{\geq 0}}|\bar{F}_{m,l}(\phi, \tau)|^2 \sinh{\tau} d\tau d\phi < \infty 
$$
where $D$ is the fundamental domain associated with the lattice $\Psi(\Gamma)$ and 
$$
R_{\omega_3(t)} F_{m,l} = e^{-imt}F_{m,l}. 
$$
By \eqref{eigenfunctionSU(1,1)}, $F_{m,l}$ are the eigenfunctions satisfying the linear differential equation
\begin{equation}\label{quotientpsilaplaceequation}
    \left[-\text{csch}{\tau}~\partial_\tau \sinh{\tau}\partial_\tau - \text{csch}^2{\tau}\left(\partial^2_{\phi\phi} + 2im \sinh{\tau} \partial_\phi  - m^2\right)\right] \bar{F}_{m,l} = l(l+1)\bar{F}_{m,l}.
\end{equation}
From the values of the spectral parameter $r$ in \eqref{rmeigenmodes}, the permissible values of $l$ are contained in the discrete set 
\begin{equation}\label{valuesofl}
    \mathcal{S}(\Gamma) =\left\{ l  \in \mathbb{C} :  l= -\frac{1}{2} + ir_j, \Im{l} = r_j  \geq 0; ~ 
l \in (-1,0)\backslash\left\{-\frac{1}{2}\right\};~ l = n  - \frac{1}{2}, n \in \mathbb{Z}_{> 0}\right \}.
\end{equation}

The geodesic flow on $\Psi(\Gamma)\backslash PSU(1,1)$ is the right-action of the nutation angle 
$$G^t \cdot \Psi(\Gamma) \bar{g} = \Psi(\Gamma) g\bar{g_t}; ~\bar{g_t}  \sim \begin{pmatrix}
     \cosh{\frac{t}{2}} & \sinh{\frac{t}{2}} & 0   \\
        \sinh{\frac{t}{2}} & \cosh{\frac{t}{2}} & 0 \\ 
        0 & 0 & 1 
\end{pmatrix}$$
on the quotient of the Lorentz group $SO(2,1)$ by  $\Psi(\Gamma)$ via the double-covering map $\pi: SU(1,1) \rightarrow  SO(2,1)$. The basis vectors of the non-degenerate subspace of $\mathfrak{so}(2,1)$ is generated by 
$$
J_1 = \begin{pmatrix}
     0 & 1 & 0  \\
        1 & 0 & 0 \\ 
        0 & 0 & 0
\end{pmatrix}, J_2 = \begin{pmatrix}
     0 & 0 & 1  \\
        0 & 0 & 0 \\ 
        1 & 0 & 0 \end{pmatrix}, J_3 = \begin{pmatrix}
     0 & 0 & 0  \\
        0 & 0 & 1 \\ 
        0 & -1 & 0,
\end{pmatrix}
$$ which correspond to the generators 
$$
K_1 = \frac{1}{2}\begin{pmatrix}
     0 & 1  \\
        1 & 0  
\end{pmatrix}, K_2 = \frac{1}{2}\begin{pmatrix}
     0  & -i  \\
        i & 0  \end{pmatrix}, K_3 = \frac{1}{2}\begin{pmatrix}
     i & 0   \\
        0 & -i 
\end{pmatrix}
$$
of $\mathfrak{su}(1,1)$ under the isomorphism of the Lie algebras. 
Here we noted that the action of the geodesic flow is equivalent to the right-regular transformation by $R_{\exp (\frac{t}{2}J_1)}$. From \eqref{decompl2qu2}, the Koopman operator then possesses the expansion of the regular representations
\begin{align}
    V_t 
    &= \sum_{m \in \mathbb{Z}}\left[\sum_{l_j = -\frac{1}{2} + ir_j  \in \mathcal{S}(\Gamma)}R_{\bar{g_t}}\big|_{\mathscr{H}_m^{(l_j,0)}} + \sum_{l = 1}^{|m|} R_{\bar{g_t}}\big|_{\mathscr{H}_m^{(l,0, \pm)}} \right]. \label{decompG^t}
\end{align}
For simple notation, we henceforth use $g_t$ in place of $\bar{g_t}$ to denote the geodesic action via group translation on $\Gamma\backslash G \cong \Psi(\Gamma)\backslash SO(2,1)$.

The exponentials $\exp{\frac{t}K_i}$, $i = 1,2,3$ corresponds to their $\pi$-preimages $\omega_i$. Hence, we have the associated differential operators (Lie-derivatives of the unitary irreducible transformations acting on $\mathcal{D}$ whose basis is of the form $\{e^{-k\theta}\}$) 
\begin{align*}\mathcal{L}_{J_i}(T^{l,0}) &= \frac{d}{dt}\big|_{t=0} T^{(l,0)}_{\omega_i(t) } = \lim_{t\rightarrow 0}t^{-1}\left(R_{\exp(tK_i)}\big|_{\mathscr{H}^{(l,0)}_m} - I\right)   
\end{align*}
 Define $X_\pm =\mathcal{L}_{J_3}(T^{(l,0)}) \mp \mathcal{L}_{J_2}(T^{(l,0)})$ respectively to be the corresponding lowering and raising operators analogous to \eqref{Liebracket} acting on $L^2(\Psi(\Gamma)\backslash PSU(1,1))$.  We obtain the relations 
 \begin{equation}\label{therelyouneed}
     [\mathcal{L}_{J_1}(T^{(l,0)}), X_\pm]= X_\pm,~ [X_-,X_+] = 2\mathcal{L}_{J_1}(T^{(l,0)}).
 \end{equation}
 In the realization on $\mathcal{D}$,  the derivatives with respect to the variable $\theta$ are computed from \eqref{T^chiaction} to be
 \begin{align}
     \mathcal{L}_{J_1}(T^{(l,0)}) &= \frac{1}{2} \left(le^{i\theta} + le^{-i\theta} - 2\sin(\theta) \partial_{\theta}\right), \nonumber \\
     \mathcal{L}_{J_2}(T^{(l,0)}) &= \frac{i}{2} \left( le^{i\theta} - le^{-i\theta} + 2i\cos(\theta)\partial_\theta\right),  \nonumber \\
     \mathcal{L}_{J_3}(T^{(l,0)}) &= \partial_\theta \label{liederivativeonD}
 \end{align}
whose corresponding lowering/raising and flow-generating operators analogous to \eqref{Liebracket} are respectively 
\begin{align*}
    \mathcal{L}_\pm= -\mathcal{L}_{J_1}(T^{(l,0)}) \mp i\mathcal{L}_{J_2}(T^{(l,0)}),~
    \mathcal{L}_3 = i\mathcal{L}_{J_3}(T^{(l,0)}),
\end{align*}
which satisfy the commutation relations
\begin{proposition}
    \begin{equation}\label{lowraise}
        [\mathcal{L}_3, \mathcal{L}_\pm] = \pm \mathcal{L}_\pm, ~ [\mathcal{L}_+, \mathcal{L}_-] = 2\mathcal{L}_3.
\end{equation}
\end{proposition}
\begin{proof}
The operators are readily given as
\begin{align*}
\mathcal{L}_+ &= ie^{-i\theta}\partial_\theta - le^{-i\theta},~
\mathcal{L}_- = -ie^{i\theta}\partial_\theta -le^{i\theta},~
\mathcal{L}_3 =  i\partial_\theta.
\end{align*}
The relations are obtained from straightforward computations. 

\end{proof}

We have the decomposition of $\mathcal{H}$ into direct sum of unitary irreducible representations 
\begin{equation}\label{decomp3}
    \mathcal{H} = \bigoplus_{l(l+1) \in \sigma\left(\Omega_{\mathbb{S}^1_{\Psi(\Gamma)}}\right)} \bigoplus_{j = 1}^{\mu_{\Psi(\Gamma)}(l)}  \mathcal{H}(T^{(l,0),j}) = \bigoplus_{l \in \mathcal{S}(\Gamma)} \bigoplus_{m \in  \mathcal{Z}} \mu_{\Psi(\Gamma)}(l)  \mathscr{H}_m^{(l,0)} 
\end{equation}
where $\Omega_{\mathbb{S}^1_{\Psi(\Gamma)}}\big|_{\cH(T^{(l,0)})} = -\mathcal{L}^2_{J_1}(T^{(l,0)}) - \mathcal{L}^2_{J_2}(T^{(l,0)}) + \mathcal{L}^2_{J_3}(T^{(l,0)})$  is the induced quadratic form (Casimir) on the quotient space $\mathbb{S}^1_{\Psi(\Gamma)}$, $\mu_{\Psi(\Gamma)}(l)$ denotes the multiplicities of spectral values $l(l+1)$, and $\mathcal{Z}$ is the set of integers in \eqref{decompG^t} (cf. \eqref{decomp1}). Namely, the values of $m$
run over the integers for the principal series, and for a fixed $l \in \mathcal{S}(\Gamma)$ associated with the unitary irreducible representation of the discrete series, $m$ ranges over the integers such that $|m| \geq l +\frac{1}{2}$.
\section{\centering \sc Proof of Theorem \ref{maintheoremmixing1}}

\subsection{\textbf{Spectral decomposition of the resolvent correlation function}.}
Denote by $R^l(z) : = R(z)\big|_{\mathcal{H}(T^{(l,0)})}$ the restriction of the resolvent onto the representations $\cH(T^{(l,0)})$ in \eqref{decomp3} and by $\hat{L}$ the flow-generating operator so that:  $e^{\hat{L}}\big|_{\mathcal{H}(T^{(l,0)})} = e^{\mathcal{L}_{J_1}(T^{l,0})} = V^t\big|_{\mathcal{H}(T^{l,0})}$. Define 
$$\hat{C}_l(F, G)(z) : = \langle R^l(z) F, G \rangle = \int_0^\infty e^{-zs} C_l(F,G)(s) ds,~ \Re{z} > 0,~ F,G \in \mathcal{H}(T^{(l,0)})$$ 
to be the correlation function of $R^l(z)$ and 
\begin{align*}
    C_l(F,G)(t) &:= \langle  V^t\big|_{\mathcal{H}(T^{l,0})} F,G   \rangle_{L^2(SO(2,1))} \\
    &= \frac{1}{4\pi^2}\int_{\R\geq 0}\int_{[0,2\pi]}\int_{[0,2\pi]}R_{g_t}F(g)\overline{G(g)}\sinh{\tau}d\phi d\psi d\tau.
\end{align*}
 For $v \in \cH (T^{(l,0)})$, the restricted resolvent has the integral representation  
\begin{align*}
     R^l(z)v  
     &=\int_0^\infty e^{-zs} T^{(l,0)}_{g_s} v ds.
\end{align*}
Via normalization, we obtain an orthonormal basis (weights) $\{v_{l,m}\}$ in the dual sense from the decomposition of $(\cH(T^{l,0}))^*$ into at-most-1-dimensional subspaces $(\cH_m)^*$ under the action of the subgroup $\Omega$, where $m$ is integral. Namely, these are the row vectors $(T_{\bullet}^{(l,0)})_{m\cdot}$ (cf. \eqref{corresDtoL^2}) where the columm index ranges over the integers. With respect to this basis, 
\begin{align*}
    \hat{C}_l(F, G)(z)  &= \sum_{m,n \in \mathcal{Z}} \langle F, v_{l,m}\rangle  \langle v_{l,n}, G \rangle \int_0^\infty e^{-zs} \langle T^{(l,0)}_{g_s} v_{l,m}, v_{l,n} \rangle ds
\end{align*}
for $F,G \in \mathcal{H}(T^{(l,0)})$. From \eqref{matrixele2},
\begin{align}
    &\langle T^{(l,0)}_{g_s} v_{l,m}, v_{l,n} \rangle = (T^{(l,0)}_{g_s})_{mn} =  B^l_{mn}(\cosh{s}) \nonumber  \\
    &= \frac{1}{2\pi}\int_0^{2\pi}\sum_{k_1, k_2 \geq  0} \begin{pmatrix}
        l+n \\ k_1
    \end{pmatrix} \begin{pmatrix}
        l-n\\ k_2
    \end{pmatrix} \cosh^{2l - (k_1 + k_2)} \frac{s}{2} \sinh^{k_1 + k_2} \frac{s}{2}  \nonumber \\
    &~~~~~~~~~~~~~~~~~~~~~~~~~~~~~~~~~~~~~~~~~~~~~~~~~~~~~~~~~~~~~~~~~~~~~~e^{i(k_1-k_2+m-n)\theta}d\theta  \nonumber\\
    &= \sum_{k_1 \geq 0}\begin{pmatrix}
        l +n \\ k_1 
    \end{pmatrix} \begin{pmatrix}
        l -n \\ k_1  + m -n
    \end{pmatrix} \cosh^{2l - 2k_1 + n - m}\frac{s}{2}\sinh^{2k_1+m-n}\frac{s}{2}  \nonumber\\
    &= \frac{1}{2^{2l}}\sum_{k_1, p,q \geq 0 }\begin{pmatrix}
        l+n \\ k_1
    \end{pmatrix} \begin{pmatrix}
        l -n \\ k_1 + m - n
    \end{pmatrix} \begin{pmatrix}
        2l - 2k_1 + n -m \\p
    \end{pmatrix}\begin{pmatrix}
        2k_1- n +m \\q \end{pmatrix} \nonumber\\
&~~~~~~~~~~~~~~~~~~~~~~~~~~~~~~~~~~~~~~~~~~~~~~~~~~~~~~~~~~~~~~~~~~~~~~(-1)^qe^{(l - p - q)s}. \label{B(ch(s)}
\end{align}
Let $\sigma(l) := \{\lambda = l_\pm - k: k \in \mathbb{Z}_{\geq 0}\}$. Since $\overline{B^l_{mn}(\cosh{s})} = B^{\bar{l}}_{mn}(\cosh{s})$, the expression in \eqref{B(ch(s)} can be rewritten to obtain
\begin{align}
     \hat{C}_l(F, G)(z)  
     &= \sum_{m,n}\sum_{\lambda \in \sigma(l)} \langle F, v_{l,m}\rangle \overline{ \langle G, v_{l,n} \rangle} \frac{\gamma^l_{m,n}(\lambda)}{\lambda - z}.\label{corr1} 
\end{align}
Denote by $v^*_{l,m} = \langle \bullet , v_{l,m}\rangle$ to be the dual basis element. Additionally, we consider the dense subset of finitely-generated functions 
\begin{align}
\mathcal{T} &= \bigg\{
F^{(M)} :=\sum_{l \in \mathcal{S}(\Gamma)} \sum_{j = 1}^{\mu(\Psi(\Gamma))(l)} \sum_{\#\{m\} \leq M} a_{l,m,j} v_{l,m,j},  M \in \mathbb{Z}_{>0} \bigg\} \subset \mathcal{D}(\Psi(\Gamma)\backslash PSU(1,1)). \label{theset}
\end{align}
\begin{proposition}\label{decompofresolvent}Let $F,G \in \mathcal{T}$. For each $l \in \mathcal{S}(\Gamma)$, $\hat{C}_l(F, G)(z)$ is a  meromorphic function on $\{\Re{(z)} > |\Re{(l)}|\}$ given by the formula
\begin{align}
 \hat{C}_l(F, G)(z)&= \sum_{l_\pm}\left[\frac{\langle F, \varphi_{l_\pm} \rangle\overline{\langle G, \varphi_{-\bar{l}_\pm}\rangle}}{l_\pm -z} + \sum_{k \in \mathbb{Z}_{>0}} \frac{\langle F, \varphi_{l_\pm - k} \rangle\overline{\langle G, \varphi_{-\bar{l}_\pm + k}\rangle} }{l_\pm - k - z}\right]\label{decompoftheresolvent}
\end{align}
where 
\begin{align*}
   &\bullet ~~\varphi_{l_+}^*  = \sum_{m\in \mathcal{Z}} \frac{1}{\Gamma(l - m +1)\Gamma(l + m + 1)} v^*_{l,m},~~   \varphi_{-\overline{l_+}}^* = \sum_{n\in \mathcal{Z}} \frac{\Gamma^2(l+1)}{2^{2l}\sqrt{\pi}} v^*_{l,n}, \\
    & \varphi_{l_-}^* = \sum_{m\in \mathcal{Z}} (-1)^m v^*_{l,m},~~   
 \varphi_{-\overline{l_-}}^* = \sum_{n\in \mathcal{Z}} \frac{\Gamma^2(l+1)\Gamma(l - n +1)\Gamma(l + n + 1)}{2^{2\bar{l}}\sqrt{\pi}} v^*_{l,n},  \\
 &\bullet ~~ \varphi_{l_\pm - k}^* =  \varphi_{l_\pm}^* X^k_+,~~ \varphi_{-\overline{l_\pm} + k} = \varphi^*_{-\overline{l_\pm}}X^k_-.
\end{align*}
\end{proposition}
\begin{proof}

For each $\lambda$, we define $\mathcal{P}^l_\lambda: \mathcal{H}  \rightarrow \mathcal{H}^*$ to be the distribution-valued operator
$$ G \mapsto (\bullet \mapsto 2\pi i \text{Res}(\lambda, \hat{C}_l(\bullet, G))).$$
Since the poles do not depend on the choice of $F$ and $G$, $\mathcal{P}^l_\lambda$ is well-defined.
Let us consider the contour integration of $\hat{C}_l(F, G; t)(z) := e^{zt}\hat{C}_l(F,G)(z)$ along the circle $C_R$ of large radius $R$. By \eqref{corr1}, the poles of $\hat{C}_l(F,G)$ are precisely the elements in $\sigma(l)$ and hence 
\begin{align}
   \oint_{C_R}e^{zt}\hat{C}_l(F,G)(z) dz &= 2\pi i\sum_{\lambda \in \sigma(l): |\lambda| < R} e^{\lambda t} \text{Res}(\lambda, \hat{C}_l(F,G)). \label{convergentseries}
\end{align}
By the description of $\mathcal{T}$, it is clear that the series in \eqref{convergentseries} converges for $R \rightarrow \infty$. Let $L_R$ be the closed rectangular contour containing $C_R$ such that: if $\lambda = l - i' \in \text{Int}(C_R)$ for $i' = 1, \ldots, k$, then $l_\pm - k' \notin \text{Int}(L_R)$ for all $k' > k$. 
Since $\sigma(l) \subset \{\Re{z} \leq C \}$ ($C = -\frac{1}{2}, n$ respectively for the principal and the discrete series) and  the function $\hat{C}_l(F,G; t)$ is analytic on $\text{Int}(L_R) \backslash \text{Int}(C_R)$, by deformation principle, 
\begin{align}
    C_l(F,G)(t) &= \lim_{\alpha \rightarrow \infty} \int_{L_\alpha} \hat{C}_l(F,G;t)(z) dz  \nonumber \\
    &= \lim_{R\rightarrow \infty}\oint_{C_R} \hat{C}_l(F,G; t)(z) dz \nonumber \\
    &= \sum_{\lambda \in \sigma(l)} e^{\lambda t} \mathcal{P}^l_{\lambda}(G)F \label{C_lassum}
\end{align}
where the first equality is acquired for any spectral element $l$. Indeed, for a fixed  arbitrary value $\alpha > 0$, we consider the vertical line $L_\alpha := \{\text{Re}(z) = \alpha >0 \}$.
By integrating along $\{ -\infty < \beta <\infty\}$,
\begin{align*} 
  \int_{L_a} \C_l(F,G;t)(z) dz 
  &= \int_0^\infty \int_\R e^{-(\alpha +i\beta)(s-t)}C_l(F,G)(s) d\beta ds \notag \\
  &= \int_0^\infty \delta(s-t)C_l(F,G)(s) ds = C_l(F,G)(t).
\end{align*}
Thereby, \eqref{C_lassum} gives
\begin{equation}\label{corr2}
    \hat{C}_l(F,G)(z) = \sum_{\lambda\in\sigma(l)}\frac{\mathcal{P}^l_\lambda(G) F}{\lambda - z}
\end{equation}
Comparing \eqref{corr1} and \eqref{corr2} yields 
\begin{equation}\label{P^l1}
    \mathcal{P}^l_{\lambda} = \sum_{m,n} \gamma^l_{m,n}(\lambda) v^*_{l,m} \otimes v^\dagger_{l,n}
\end{equation}
where $v^\dagger_{l,m} = \langle v_{l,n}, \bullet \rangle $. 
The operators in \eqref{lowraise} provide the linear functional relations
\begin{equation}\label{recurrent}
\begin{cases}
     v_{l,m}^*X_- &= -i(l+m)v^*_{l,m-1} - imv^*_{l,m} + \frac{i}{2}(l-m)v^*_{l,m+1}, \nonumber \\
    v_{l,m}^*X_+ &= i(l+m)v^*_{l,m-1} - imv^*_{l,m} -  \frac{i}{2}(l-m)v^*_{l,m+1},\nonumber \\
    v_{l,m}^*X_-X_+ &=   \frac{1}{2}(l + m -1)(l+m)v^*_{l,m-2} +(l+m)v^*_{l,m-1} - (l^2 + l + m^2)v^*_{l,m } + \\
    & ~~ (l-m)v^*_{l,m+1} + \frac{1}{2}(l - m -1)(l - m) v^*_{l,m-2}.   
\end{cases}
\end{equation}
Here the matrix elements $v_{l,m}^*X_\pm = \langle X_\pm (\bullet) , e^{-im\theta}  \rangle $ are calculated from the descriptions in \eqref{liederivativeonD}. 
Since $X_-\hat{L} = (\hat{L}+ I) X_-$,  
\begin{equation}\label{H-1}
X_-V^t\big|_{\mathcal{H}(T^{l,0})} - e^tV^t\big|_{\mathcal{H}(T^{l,0})} X_- = X_-T^{(l,0)}_{g_t} - e^tT^{(l,0)}_{g_t}X_- = 0.  
\end{equation}
Applying \eqref{H-1} and the equality $X_-v_{l,m} = v_{l,m+1}$ gives
\begin{align}
\langle  X_-R^l(z)v_{l,n},v_{l,m}\rangle &=   \int_0^\infty  e^{-s(z-1)}\langle  T^{(l,0)}_{g_s} v_{l,n},v_{l,m-1}\rangle ds = \sum_{\lambda \in \sigma(l)}\frac{\gamma^l_{m-1,n}(\lambda)}{z- (\lambda + 1)}  \label{H-2}
\end{align}
and hence from \eqref{corr1}, one can write
\begin{align*}
    X_-\hat{C}_l(F,G)(z) &=\sum_{m,n} \sum_{\lambda\in\sigma(l)} v_{l,m}^*(F)  \overline{ v_{l,n}^*(G)} \frac{\gamma^l_{m-1,n}(\lambda)}{z- (\lambda + 1)} \\
    &= \sum_{m,n} \sum_{\lambda\in\sigma(l)} v_{l,m}^*(F)  \overline{ v_{l,n}^*(G)} \frac{\gamma^l_{m,n}(\lambda +1)}{z- (\lambda + 1)} \\
&=\sum_{\lambda\in\sigma(l)}\frac{\mathcal{P}^l_{\lambda+1}(G) F}{z - (\lambda + 1)}.
\end{align*}
where the second equality follows from rearranging the indices in \eqref{B(ch(s)}. By comparing this expression with $X_-$ applied to \eqref{corr2}, the last equality in \eqref{H-2} suggests that, for the pole $\lambda = l_\pm$, 
\begin{align}
X_-\mathcal{P}^l_\lambda &\in \text{Im}(\mathcal{P}^l_{\lambda +1}) = 0 \label{H-nullsP}
\end{align}
since $\sigma(l)$ does not contain $l_\pm + 1$. 

Suppose $\varphi^*_{l_\pm} \in \text{Im}(\mathcal{P}^l_{l_\pm}) \subset \cH(T^{(l,0)})^*$, where   
\begin{align*}
\varphi^*_{l_\pm} = \sum_{m \in \mathcal{Z}} v^*_{l,m}(\varphi_{l_\pm}) v^*_{l,m}, 
\end{align*}
is an eigenfunctional for the flow-generating operator satisfying \begin{align*}
\begin{cases}
    \varphi^*_{l_\pm} V^t  &= T^{(l,0)}_{\omega_3}\varphi^*_{l_\pm} = \lambda \varphi^*_{l_\pm}, \\
 \varphi^*_{l_\pm} X_- &= 0
\end{cases}
\end{align*}
(from \eqref{H-nullsP}). By the functional relations and $[X_+, X_-] = 2\hat{L}$, 
\begin{align*}
& \frac{1}{2}(l + m -1)(l+m)v^*_{l,m-2}(\varphi_{l_\pm}) +(l+m)v^*_{l,m-1}(\varphi_{l_\pm}) - (l^2 + l + m^2)v^*_{l,m }(\varphi_{l_\pm}) + \\ 
 &(l-m)v^*_{l,m+1}(\varphi_{l_\pm}) + \frac{1}{2}(l - m -1)(l - m) v^*_{l,m-2}(\varphi_{l_\pm}) = -2\lambda v^*_{l,m}(\varphi_{l_\pm}),  \\
 &-i(l+m)v^*_{l,m-1}(\varphi_{l_\pm}) - imv^*_{l,m}(\varphi_{l_\pm}) + \frac{i}{2}(l-m)v^*_{l,m+1}(\varphi_{l_\pm}) = 0.\label{recurrel}
\end{align*}
The unique non-trivial solutions to the recurrence given by 
\begin{equation*}
    v^*_{l,m}(\varphi_{l_+})  =\frac{1}{\Gamma(l)\Gamma(l + m + 1)}, ~   v^*_{l,m}(\varphi_{l_-})  = (-1)^m
\end{equation*}
exist for only when $\lambda = l_+$ and $\lambda = l_-$, respectively. This shows that for any positive integer $k$, the sequence 
\begin{equation}\label{exactsequence}
    \ldots \xrightarrow{X_-} \text{Im}(\mathcal{P}^l_{l_\pm-k}) \xrightarrow{X_-} \text{Im}(\mathcal{P}^l_{l_\pm-k + 1}) \xrightarrow{X_-} \ldots \xrightarrow{X_-}  \text{Im}(\mathcal{P}^l_{l_\pm})  \xrightarrow{X_-} 0
\end{equation}
is exact and $\text{Im}(\mathcal{P}^l_{l_\pm - k})$ are respectively the 1-dimensional functional spaces
spanned by 
\begin{equation}\label{raisingweights}
\varphi_{l_+}^* X^k_+ = \sum_{m\in \mathcal{Z}} \frac{1}{\Gamma(l - m +1)\Gamma(l + m + 1)} v^*_{l,m}X^k_+,~~   \varphi_{l_-}^*X^k_+ = \sum_{m\in \mathcal{Z}} (-1)^m v^*_{l,m}X^k_+.
\end{equation}
Furthermore, if $l = -\frac{1}{2}+ ir$, we have $\bar{l} = -(l+1)$. We also use the reflection formula $\Gamma(z)\Gamma(1-z) = \frac{\pi}{\sin{\pi z }}$ to obtain the asymptotic
\begin{align}
    B^l_{m0}(\cosh{s})   
 &= \Gamma^2(l + 1) \cosh^{2l} \frac{s}{2} \coth^m \frac{s}{2} \times \nonumber \\
    &\sum_{q = 0 \vee m }^\infty \frac{\tanh^{2q}\frac{s}{2}}{\Gamma(q + 1)\Gamma(l - q )\Gamma(q + 1 - m)\Gamma(l + m - q +1)}  \nonumber \\
&= \left[\frac{1}{2^{2l}\sqrt{\pi}}\frac{\Gamma^2(l+1)}{\Gamma(l - m +1)\Gamma(l + m + 1)}e^{ls} + \frac{(-1)^m\Gamma(l+1)}{2^{2\bar{l}} \sqrt{\pi}}e^{\bar{l}s}\right] (1 + O(e^{-s})) \nonumber. 
\end{align}
for sufficiently large $s>0$. By symmetry from the Legendre function $B^n_l(\cosh{s})$, 
\begin{align*}
    B^l_{0n}(\cosh{s}) &= \frac{\Gamma(l - n +1)\Gamma(l +  n + 1)}{\Gamma^2(l+1)}B^l_{n0}(\cosh{s})  \nonumber  \\
    &= \left[\frac{1}{2^{2l}\sqrt{\pi}} e^{ls} + \frac{(-1)^m\Gamma(l+n+1)\Gamma(l - n + 1)}{2^{2\bar{l}}\sqrt{\pi} \Gamma(l +1) }e^{\bar{l}s}\right] (1 + O(e^{-s})). \label{B^l_0m}
\end{align*}
Thereby, we deduce from the product $\gamma_{mn}^l(\lambda) = v^*_{l,m}(\varphi_{\lambda})v^\dagger_{l,n}(\varphi_{-\overline{\lambda}}) =  v^*_{l,m}(\varphi_\lambda)\overline{v^*_{l,n}(\varphi_{-\overline{\lambda}})}$ the expansions
\begin{equation}
 \varphi_{-\overline{l_+}}^* = \sum_{n\in \mathcal{Z}} \frac{\Gamma^2(l+1)}{2^{2l}\sqrt{\pi}} v^*_{l,n},~~   
 \varphi_{-\overline{l_-}}^* = \sum_{n\in \mathcal{Z}} \frac{\Gamma^2(l+1)\Gamma(l - n +1)\Gamma(l + n + 1)}{2^{2\bar{l}}\sqrt{\pi}} v^*_{l,n}. 
\end{equation}
Subsequently, by \eqref{exactsequence}, the 1-dimensional functional vector space $\text{Im}(\mathcal{P}^l_{-\overline{l_\pm} + k})$ are respectively spanned by 
\begin{equation}\label{lowerweights}
     \varphi_{-\overline{l_+}}^* X^k_-,~~ \varphi_{-\overline{l_-}}^*X^k_-. 
\end{equation}
From this, we can write \eqref{P^l1} as 
\begin{equation}
  \mathcal{P}^l_{l_\pm - k} = \varphi^*_{l_\pm - k } \otimes \varphi^\dagger_{-\overline{l_\pm} + k}.\label{P^l}
\end{equation} 
\end{proof}
\subsection{\textbf{Proof of Theorem \ref{maintheoremmixing2}}.} 
Consider any elements $F,G \in \cH(T^{(l,0)})$ contained the space of exponentially decaying test functions given in terms of the basis 
\begin{equation}
    F = \sum_{m} \langle F, v_{l,m} \rangle v_{l,m}, ~G = \sum_{n}\langle G, v_{l,n} \rangle v_{l,n} 
\end{equation}
whereby $|\langle F, v_{l,m} \rangle|m^\alpha \rightarrow 0$, ~$|\langle G, v_{l,n} \rangle|n^\beta \rightarrow 0$  as $|m|\wedge |n| \rightarrow 0$ for any  $\alpha, \beta \in \R_{>0}$. In light of \eqref{C_lassum}, \eqref{raisingweights}, \eqref{lowerweights}, \eqref{P^l} and the linear functional relations, the restricted correlation function written as the sum 
\begin{align}
    C_l(F,G)(t) &= \sum_{\lambda \in \sigma(l)}\langle F, \varphi_{\lambda}\rangle \overline{\langle G, \varphi_{-\bar{\lambda}} \rangle}  e^{\lambda t} \nonumber \\
    &= \sum_{\lambda \in \sigma(l)} \sum_{m,n} \langle F, v_{l,m} \rangle \langle  v_{l,m}, \varphi_\lambda \rangle \overline{\langle G, v_{l,n} \rangle} ~ \overline{\langle v_{l,n}, \varphi_{-\bar{\lambda}} \rangle}  e^{\lambda t} \nonumber \\
    &= \sum_{k \in \mathbb{Z}_{\geq 0}} \sum_{m,n} \varphi^*_{l_\pm}X^k_+(v_{l,m}) \overline{\varphi^*_{l_\pm}X^k_-(v_{l,n})} \langle F, v_{l,m} \rangle \overline{\langle G, v_{l,n} \rangle} e^{(l_\pm - k)t}\nonumber \\
    &= \sum_{k \in \mathbb{Z}_{\geq 0}} \sum_{m,n} O(m^kn^k) \langle F, v_{l,m} \rangle \overline{\langle G, v_{l,n} \rangle} e^{(l_\pm - k)t}
\end{align}
converges absolutely for any fixed time $t \in \R_{\geq 0}$. By density, the absolute convergence holds for any pairs of  $L^2$-functions in $\cH(T^{(l,0)})$.   
From \eqref{decompG^t}, we obtain the expression for the general correlation function 
\begin{equation}\label{decompforC}
C( \bullet, \bullet )(t) = \sum_{l \in \mathcal{S}(\Gamma)} \sum_{\lambda \in \sigma(l)} \sum_{j = 1}^{\mu(\Psi(\Gamma))(l)} e^{\lambda t}  \varphi^*_{\lambda, j} \otimes \varphi^\dagger_{-\bar{\lambda}, j} .
\end{equation}
For any $t \geq 0$, we claim that $C( \bullet, \bullet) \in \mathcal{D}'(\Psi(\Gamma)\backslash PSU(1,1))\otimes \mathcal{D}'(\Psi(\Gamma)\backslash PSU(1,1))$ is a bounded linear operator on the dense subspace of finitely generated $L^2$-functions, hence implies the series in \eqref{decompforC} converges. Indeed, we consider $F^{(M)}, G^{(N)} \in \mathcal{T}$. 
By \eqref{unwrapping} and recall that $\overline{v}_{l,m,j}, \overline{v}_{l,n,j}$ are the eigenfunctions to the $\Omega$-quotient Laplace equation in \eqref{quotientpsilaplaceequation}, 
\begin{align*}
    F^{(M)} (\phi, \tau, \psi)&=  \sum_{\#\{m\} \leq M}\overline{F^{(M)}}(\phi, \tau)e^{-im\psi} \\
    &=  \sum_{\#\{m\} \leq M} \sum_{l \in \mathcal{S}(\Gamma)} \sum_{j = 1}^{\mu(\Psi(\Gamma))(l)} \langle \overline{F^{(M)}}, \overline{v}_{l,m,j}\rangle_{L^2(D)} \overline{v}_{l,m,j}(\phi, \tau) e^{-im\psi},   \\
    G^{(N)} (\phi, \tau, \psi) &=   \sum_{\#\{n\} \leq N}\overline{G^{(N)}}(\phi, \tau)e^{-in\psi} \\
    &= \sum_{\#\{n\} \leq N} \sum_{l \in \mathcal{S}(\Gamma)} \sum_{j = 1}^{\mu(\Psi(\Gamma))(l)} \langle \overline{G^{(N)}}, \overline{v}_{l,n,j}\rangle_{L^2(D)} \overline{v}_{l,n,j}(\phi, \tau) e^{-in\psi}. 
\end{align*}
Denote $a_{l,m,j}  =\langle \overline{F^{(M)}}, \overline{v}_{l,m,j}\rangle_{L^2(D)}, b_{l,n,j} =  \langle \overline{G^{(N)}}, \overline{v}_{l,n,j}\rangle_{L^2(D)} $ and we further deduce that . We have   
\begin{align*}
    |C( F^{(M)},G^{(N)} )(t)| & \leq \sum_{m,n}  \sum_{l \in \mathcal{S}(\Gamma)} \sum_{j = 1}^{\mu(\Psi(\Gamma))(l)} \left|\sum_{\lambda \in \sigma(l)}  e^{\lambda t} \varphi^*_{\lambda, j} (v_{l,m,j}) \varphi^\dagger_{-\bar{\lambda}, j}(v_{l,n,j}) \right| \left| a_{l,m,j}\overline{b_{l,n,j}} \right| \\
    &  =  \sum_{m,n} \sum_{l \in \mathcal{S}(\Gamma)} \sum_{j = 1}^{\mu(\Psi(\Gamma))(l)} |B^l_{mn}(\cosh{t})| \left| a_{l,m,j}\overline{b_{l,n,j}} \right|\\
    &\leq  C ||\overline{F^{(M)}}||_{L^2(D)} ||\overline{G^{(N)}}||_{L^2(D)}.
\end{align*}
The last inequality follows from the previous symmetry relation of the Jacobi functions and its generating function 
$$\Phi(1, t) = \sum_{m\in\mathbb{Z} \cup \mathbb{Z}+\frac{1}{2}}B^l_{mn}(\cosh{t}).$$
\qed
\subsection{\textbf{Proof of Theorem \ref{maintheoremmixing1}}.} Let $\gamma$ be a closed geodesic of period $T^\#_\gamma$ starting at an element $g\in \Psi(\Gamma)\backslash PSU(1,1)$ fixed by $G^t$. The differential of the flow $dG^{T^\#_\gamma}$ is given by the adjoint action associated with the $\mathfrak{su}(1,1)$-generators $K_i$, $i=1,2,3$ evaluated at $t = T^\#_\gamma$. The respective images of the tangential vectors on the transversal surface to the geodesic in $E^s(g)$, $E^u(g)$ and $E^0(g)$
\begin{align}
    dG^{T^\#_\gamma} (K_2) = \text{Ad}_{g_{T^\#_\gamma}}(K_2) &= \sinh{\frac{t}{2}}K_1 + \cosh{\frac{t}{2}} K_2  = e^{-T^\#_\gamma}K_2, \nonumber \\
    dG^{T^\#_\gamma} (K_1) = \text{Ad}_{g_{T^\#_\gamma}}(K_1) &= \cosh{\frac{t}{2}}K_2 + \sinh{(\frac{t}{2})}K_1 =  e^{T^\#_\gamma}K_1 , \nonumber \\
    dG^{T^\#_\gamma} (K_3) = \text{Ad}_{g_{T^\#_\gamma}}(K_3) &= K_3. \label{P=dG}
\end{align}
The linearized Poincar\'{e} map for $\gamma$ has the matrix form 
\begin{equation*}
    \mathbf{P}_\gamma = \begin{pmatrix}
        e^{T^\#_\gamma} & 0 \\ 0 & e^{-T^\#_\gamma}
    \end{pmatrix}
\end{equation*}
Hence, $|\det(I - \mathbb{P}_\gamma)| = (1- e^{T^\#_\gamma})(1- e^{-T^\#_\gamma}) = -2\sinh{T^\#_\gamma}$. From the spectral decomposition in \eqref{decompforC}, periodicity requires $e^{\lambda T^\#_\gamma} = 1$, hence implies $\Im{(\lambda)} T^\#_\gamma = 2\pi k'$, $k' \in \mathbb{Z}$ and that 
$$
T^\#_\gamma = \frac{2\pi}{r_j}
$$
for the eigenvalues $\lambda = l_\pm - k$, $l_\pm (j) = -\frac{1}{2} \pm ir_j$ in $\mathcal{S}(\Gamma)$ associated with unitary irreducible representations of the principal series. This shows that $G^t$ is Lefschetz, hence we obtain the flat trace distribution 
\begin{equation}\label{flattrace2}
    \textup{Tr}^\flat V^t = \sum_{r_j\in \Im{(S(\Gamma))}}\sum_{n = 1}^\infty \frac{2\pi}{r_j\sinh{\frac{n\pi}{r_j}}}\delta\left(t - \frac{2n\pi}{r_j}\right).
\end{equation}
\qed


{\centering}

\end{document}